\documentclass[letterpaper, 10 pt, conference]{ieeeconf}      

\IEEEoverridecommandlockouts                              

\usepackage{graphics} 
\usepackage{epsfig} 
\usepackage{times} 
\usepackage{amsmath} 
\usepackage{amssymb}  
\usepackage{subfigure}
\usepackage{url}
\usepackage{dsfont}
\newtheorem{theorem}{Theorem}

\newtheorem{algorithm}[theorem]{Algorithm}

\usepackage{cite}

\usepackage{color}
\usepackage{algorithm}
\newtheorem{properties}[theorem]{Properties}

\newtheorem{definition}[theorem]{Definition}

\newtheorem{proposition}[theorem]{Proposition}

{ 
\newtheorem{remark}[theorem]{Remark}

}

\newcommand{\bi}{\begin{itemize}}
\newcommand{\ei}{\end{itemize}}
\newcommand{\bd}{\begin{displaymath}}
\newcommand{\ed}{\end{displaymath}}
\newcommand{\be}{\begin{eqnarray*}}
\newcommand{\ee}{\end{eqnarray*}}

\usepackage[normalem]{ulem}

\title{\LARGE \bf
On Computation of Koopman Operator from Sparse Data}
\author{Subhrajit Sinha and  Enoch Yeung\\
\thanks{S. Sinha is with Pacific Northwest National Laboratory and E. Yeung is with the Department of Mechanical Engineering at University of California, Santa Barbara. \tt \small {email : subhrajit.sinha@pnnl.gov}
}
}

\begin{document}
\maketitle

\begin{abstract}
In this paper we propose a novel approach to compute the Koopman operator from sparse time series data. In recent years there has been considerable interests in operator theoretic methods for data-driven analysis of dynamical systems. Existing techniques for the approximation of the Koopman operator require sufficiently large data sets, but in many applications, the data set may not be large enough to approximate the operators to acceptable limits. In this paper, using ideas from robust optimization, we propose an algorithm to compute the Koopman operator from sparse data. We enrich the sparse data set with artificial data points, generated by adding bounded artificial noise and and formulate the noisy robust learning problem as a robust optimization problem and show that the optimal solution is the Koopman operator with smallest error. We illustrate the efficiency of our proposed approach in three different dynamical systems, namely, a linear system, a nonlinear system and a dynamical system governed by a partial differential equation. 
\end{abstract}

\section{Introduction}
In recent years there has been an increasing number of fields where the underlying physical systems are represented by data. Increasing memory, processor speed, and the advent of distributed computing architectures have enabled us to handle and analyze data with more precision and to address learning problems for nonlinear complex systems at an unprecedented scale. Another advantage of data-driven analysis lies in the fact that for many naturally occurring complex systems and engineered systems with emergent phenomena, e.g. biological systems, inter-dependent critical infrastructure, social networks, financial systems, it may not always be possible to derive and analyze theoretical mathematical models of the underlying systems \cite{yeung2015global}. In such cases, one has to resort to data-driven techniques for understanding the behavior of such systems.

In the realm of dynamical systems, there is an increasing interest in transfer operator theoretic techniques for analysis and control \cite{yeung2018koopman,yeung2017learning, Dellnitz_Junge,Mezic2000,froyland_extracting,Junge_Osinga,Mezic_comparison,
Dellnitztransport,mezic2005spectral,Mehta_comparsion_cdc,Vaidya_TAC,
raghunathan2014optimal,susuki2011nonlinear,mezic_koopmanism,
mezic_koopman_stability,surana_observer}. The operator based techniques differ from the classical study of dynamical systems in the sense that instead of studying the dynamical system on the configuration manifold and the tangent and cotangent bundles associated with them, the evolution is studied on the space of functions \cite{johnson2018class} or on the space of measures. The big advantage of using the function space or the measure space instead of the configuration space is the fact that in those spaces, the evolution is linear, even if the actual system is nonlinear. Moreover, the operators, namely, Perron-Frobenius operator and the Koopman operator are positive Markov operators. These properties can be exploited to have probabilistic interpretations and can be used for various applications like the optimal placement of sensors and actuators \cite{optimal_placement_ECC, optimal_placement_JMAA} etc. However, the trade-off lies in the fact that the operators defined on the function space or the measure space are infinite dimensional. 

Another advantage of the operator theoretic framework for analysis and control of dynamical systems is that the methods developed are data-centric and in many cases wholly data-driven. In particular, a finite-dimensional approximation of both Perron-Frobenius and Koopman operators can be constructed from time-series data obtained from experiments \cite{yeung2017learning}. Towards this goal various data-driven methods are proposed for the finite dimensional approximation of these operators \cite{dellnitz2002set, Mezic2000,DMD_schmitt,rowley2009spectral,EDMD_williams}, with Dynamic Mode Decomposition (DMD) and extended DMD being the ones which are used extensively. Recent works have also addressed the problem of computing these operators for systems with process and observation noise and for Random Dynamical Systems (RDS)  \cite{mezic_stochastic_koopman_spectrum,PhysRevE.96.033310, robust_DMD_ACC,robust_DMD_arxiv}. In \cite{mezic_stochastic_koopman_spectrum} the authors have provided a characterization of the spectrum and eigenfunctions of the Koopman operator for discrete and continuous time RDS, while in \cite{PhysRevE.96.033310}, the authors have provided an algorithm to compute the Koopman operator for systems with both process and observation noise. In \cite{robust_DMD_ACC,robust_DMD_arxiv} the authors used robust optimization-based techniques to compute the approximate Koopman operator for data sets of finite length and have shown that normal DMD or EDMD and subspace DMD \cite{PhysRevE.96.033310} lead to an unsatisfactory approximation of Koopman operator for data sets of finite length.

A different challenge that researchers often have to account for is the scenario when the data set is not only finite, but also small, that is sparse. Sparse data refers to
data sets with few data points and can have an immense effect on the ability to train the
Koopman operator into producing accurate predictions. In particular, in the case of sparse data the existing DMD and EDMD algorithms may lead to an ill-conditioned least-square problem.
In this paper, we address the problem of computation of Koopman operator when the data set is sparse. We append artificial data points to the sparse data set to enrich the data and use robust optimization-based techniques to obtain the approximate Koopman operator. The robust optimization problem is a min-max problem which can be approximated as a least squares problem with a regularization term. The regularization parameter imposes sparsity in the Koopman operator. Moreover, it prevents over-fitting of the data and hence can be used to design a data-driven predictor \cite{korda_mezic_predictor}.

The organization of the paper is as follows. In section \ref{section_transfer_operators} we provide the basics of transfer operators followed by a discussion of DMD and EDMD algorithms in section \ref{section_DMD}. In section \ref{section_sparse_Koopman} we present the main results of the paper and state the algorithm to construct the Koopman operator for sparse data. Design of the robust predictor is discussed in section \ref{section_predictor}, with simulation results in section \ref{section_simulation}. Finally we conclude the paper in \ref{section_conclusion}.

\section{Transfer Operators for Dynamical Systems}\label{section_transfer_operators}
Consider a discrete-time dynamical system
\begin{eqnarray}\label{system}
z_{t+1} = T(z_t)
\end{eqnarray}
where $T:Z\subset \mathbb{R}^N \to Z$ is assumed to be an invertible smooth diffeomorphism.  Associated with the dynamical system (\ref{system}) is the Borel-$\sigma$ algebra ${\cal B}(Z)$ on $Z$ and the vector space ${\cal M}(Z)$ of bounded complex valued measures on $X$. With this, two linear operators, namely, Perron-Frobenius (P-F) and Koopman operator, can be defined as follows \cite{Lasota} :
\begin{definition}[Perron-Frobenius Operator] 
The Perrorn-Frobenius operator $\mathbb{P}:{\cal M}(Z)\to {\cal M}(Z)$ is given by
\[[\mathbb{P}\mu](A)=\int_{{\cal Z} }\delta_{T(z)}(A)d\mu(z)=\mu(T^{-1}(A))\]
$\delta_{T(z)}(A)$ is stochastic transition function which measure the probability that point $z$ will reach the set $A$ in one time step under the system mapping $T$. 
\end{definition}

\begin{definition}[Invariant measures] Invariant measures are the fixed points of
the P-F operator $\mathbb{P}$ that are also probability measures. Let $\bar \mu$ be the invariant measure then, $\bar \mu$ satisfies
\[\mathbb{P}\bar \mu=\bar \mu.\]
\end{definition}

If the state space $Z$ is compact, it is known that the P-F operator admits at least one invariant measure.

\begin{definition} [Koopman Operator] 
Given any $h\in\cal{F}$, $\mathbb{U}:{\cal F}\to {\cal F}$ is defined by
\[[\mathbb{U} h](z)=h(T(z))\]
where $\cal F$ is the space of function (observables) invariant under the action of the Koopman operator.
\end{definition}

Both the P-F operator and the Koopman operator are linear operators, even if the underlying system is non-linear. But while analysis is made tractable by linearity, the trade-off is that these operators are typically infinite dimensional. In particular, the P-F operator and Koopman operator often will lift a dynamical system from a finite-dimensional space to generate an infinite dimensional linear system in infinite dimensions. 
\begin{properties}\label{property}
Following properties for the Koopman and Perron-Frobenius operators can be stated \cite{Lasota}.

\begin{enumerate}
\item [a).] For the Hilbert space ${\cal F}=L_2(Z,{\cal B}, \bar \mu)$ 
\begin{eqnarray*}
&&\parallel \mathbb{U}h\parallel^2=\int_Z |h(T(z))|^2d\bar \mu(z)
\nonumber\\&=&\int_Z | h(z)|^2 d\bar\mu(z)=\parallel h\parallel^2
\end{eqnarray*}
where $\bar \mu$ is an invariant measure. This implies that Koopman operator is unitary.

\item [b).] For any $h\geq 0$, $[\mathbb{U}h](z)\geq 0$ and hence Koopman is a positive operator.

\item [c).]For invertible system $T$, the P-F operator for the inverse system $T^{-1}:Z\to Z$ is given by $\mathbb{P}^*$ and $\mathbb{P}^*\mathbb{P}=\mathbb{P}\mathbb{P}^*=I$. Hence, the P-F operator is unitary.

\item [d).] If the P-F operator is defined to act on the space of densities i.e., $L_1(Z)$ and Koopman operator on space of $L_\infty(Z)$ functions, then it can be shown that the P-F and Koopman operators are dual to each other \footnote{with some abuse of notation we use the same notation for the P-F operator defined on the space of measure and densities.}
\begin{eqnarray*}
&&\left<\mathbb{U} f,g\right>=\int_Z [\mathbb{U} f](z)g(z)dx\nonumber\\&=&\int_Xf(y)g(T^{-1}(y))\left|\frac{dT^{-1}}{dy}\right|dy=\left<f,\mathbb{P} g\right>
\end{eqnarray*}
where $f\in L_{\infty}(Z)$ and $g\in L_1(Z)$ and the P-F operator on the space of densities $L_1(Z)$ is defined as follows
\[[\mathbb{P}g](z)=g(T^{-1}(z))|\frac{dT^{-1}(z)}{dz}|.\]

\item [e).] For $g(z)\geq 0$, $[\mathbb{P}g](z)\geq 0$.

\item [f).] Let $(Z,{\cal B},\mu)$ be the measure space where $\mu$ is a positive but not necessarily the invariant measure of $T:Z\to Z$, then the P-F operator $\mathbb{P}:L_1(Z,{\cal B},\mu)\to L_1(Z,{\cal B},\mu)$  satisfies  following property:

 \[\int_Z [\mathbb{P}g](z)d\mu(z)=\int_Z g(z)d\mu(x).\]\label{Markov_property}
\end{enumerate}
\end{properties}

\section{Dynamic Mode Decomposition (DMD) and Extended Dynamic Mode Decomposition (EDMD)}\label{section_DMD}
Dynamic Mode Decomposition was first proposed in \cite{DMD_schmitt} in the context of analysis of fluid flow analysis. DMD algorithm is a way of approximating the spectrum of the Koopman operator. DMD algorithm was improved upon (EDMD) in \cite{EDMD_williams}, where the authors used a choice of observables, called a \textit{dictionary}. In this improved algorithm, the Koopman operator is approximated as a linear map on the span of the finite set of dictionary functions. In this section, we briefly describe the EDMD algorithm for approximating the Koopman operator. 

Consider snapshots of data set obtained from simulating a discrete time dynamical system $z\mapsto T(z)$ or from an experiment
\begin{eqnarray}
X_p = [x_1,x_2,\ldots,x_M],& X_f = [y_1,y_2,\ldots,y_M] \label{data}
\end{eqnarray}
where $x_i\in X$ and $y_i\in X$. The two pair of data sets are assumed to be two consecutive snapshots i.e., $y_i=T(x_i)$. Let $\mathcal{D}=
\{\psi_1,\psi_2,\ldots,\psi_K\}$ be the set of dictionary functions or observables, where $\psi : X \to \mathbb{C}$. Let ${\cal G}_{\cal D}$ denote the span of ${\cal D}$ such that ${\cal G}_{\cal D}\subset {\cal G}$, where ${\cal G} = L_2(X,{\cal B},\mu)$. The choice of dictionary functions are very crucial and it should be rich enough to approximate the leading eigenfunctions of Koopman operator. Define vector valued function $\mathbf{\Psi}:X\to \mathbb{C}^{K}$
\begin{equation}
\mathbf{\Psi}(\boldsymbol{x}):=\begin{bmatrix}\psi_1(x) & \psi_2(x) & \cdots & \psi_K(x)\end{bmatrix}
\end{equation}
In this application, $\mathbf{\Psi}$ is the mapping from physical space to feature space. Any function $\phi,\hat{\phi}\in \mathcal{G}_{\cal D}$ can be written as
\begin{eqnarray}
\phi = \sum_{k=1}^K a_k\psi_k=\boldsymbol{\Psi^T a},\quad \hat{\phi} = \sum_{k=1}^K \hat{a}_k\psi_k=\boldsymbol{\Psi^T \hat{a}}
\end{eqnarray}
for some set of coefficients $\boldsymbol{a},\boldsymbol{\hat{a}}\in \mathbb{C}^K$. Let \[ \hat{\phi}(x)=[\mathbb{U}\phi](x)+r,\]
where $r$ is a residual function that appears because $\mathcal{G}_{\cal D}$ is not necessarily invariant to the action of the Koopman operator. To find the optimal mapping which can minimize this residual, let $\bf K$ be the finite dimensional approximation of the Koopman operator. Then the  matrix $\bf K$ is obtained as a solution of least square problem as follows 
\begin{equation}\label{edmd_op}
\min\limits_{\bf K}\parallel {\bf G}{\bf K}-{\bf A}\parallel_F
\end{equation}
\begin{eqnarray}\label{edmd1}
\begin{aligned}
& {\bf G}=\frac{1}{M}\sum_{m=1}^M \boldsymbol{\Psi}({x}_m)^\top \boldsymbol{\Psi}({x}_m)\\
& {\bf A}=\frac{1}{M}\sum_{m=1}^M \boldsymbol{\Psi}({x}_m)^\top \boldsymbol{\Psi}({y}_m),
\end{aligned}
\end{eqnarray}
with ${\bf K},{\bf G},{\bf A}\in\mathbb{C}^{K\times K}$. The optimization problem (\ref{edmd_op}) can be solved explicitly to obtain following solution for the matrix $\bf K$
\begin{eqnarray}
{\bf K}_{EDMD}={\bf G}^\dagger {\bf A}\label{EDMD_formula}
\end{eqnarray}
where ${\bf G}^{\dagger}$ is the psedoinverse of matrix $\bf G$.
DMD is a special case of EDMD algorithm with ${\bf \Psi}(x) = x$.

\section{Koopman Operator Construction for Sparse Data}\label{section_sparse_Koopman}
In many experiments, it often is the case that the data obtained is not rich enough to achieve a good enough approximation of Koopman operator. In this case, existing algorithms like DMD or EDMD fail to generate acceptable Koopman operators. In fact, in many instances, these algorithms lead to unstable eigenvalues, even though the underlying system is stable \cite{robust_DMD_ACC,robust_DMD_arxiv}. In this section, a robust optimization-based framework is presented for a better approximation of the Koopman eigenspectrum for sparse data.

\subsection{Enrichment of the existing dataset}
Let $\bar{X}_p = [{x}_1,{x}_2,\cdots , {x}_M]$ and $\bar{X}_f = [{y}_1,{y}_2,\cdots , {y}_M]$ be the snap-shots of data obtained from a discrete-time dynamical system $z\to T(z)$, where $z\in Z \subset \mathbb{R}^N$, $x_i\in Z$ and $y_i\in Z$. The two pair of data sets are assumed to be two consecutive snapshots i.e., $y_i=T(x_i)$. Consider ${x}_i+\delta x_i$, where $\parallel \delta x_i\parallel \leq \lambda_X$. Since, $T$ is a diffeoemorphism, 
\begin{eqnarray}
T(x_i+\delta x_i)\approx T(x_i)+\frac{\partial T}{\partial x}\delta x_i = y_i + \delta y_i.
\end{eqnarray}
Since $\parallel \delta x_i\parallel \leq \lambda_X$, $x_i+\delta x_i\in B(x_i,\lambda_X)$, where 
\begin{eqnarray*}
B(x_0,r) = \{x\in\mathbb{R}^N| \parallel x - x_0 \parallel \leq r\}.
\end{eqnarray*}
Again $T$ being a smooth diffeomorphism implies that $\frac{\partial T}{\partial x} \leq \lambda_T$ and hence $T(x_i+\delta x_i)\in B(y_i,\lambda_Y)$, where $\lambda_Y\leq \lambda_X\lambda_T$.

Hence, a \textit{sufficiently small perturbation} to any point $x_i$ will result in a small enough perturbation of $y_i$ and it is the boundedness of the small perturbations that is used to enrich the existing data set. In particular, to each observed data tuple $(x_i,y_i)$, we augment an extra data point $(x_i+\delta x_i, y_i+\delta y_i)$. Note that, for each data point more than one data point can be augmented, but for clarity, we will discuss the situation where only one extra artificial data point is augmented to each observed data point. Hence, an artificial data set 
\begin{eqnarray}\label{enriched_dataset}
\begin{aligned}
{X}_p &= [{x}_1,\cdots , {x}_M,x_1+\delta x_1, \cdots , x_M + \delta x_M]\\
& = [{x}_1,\cdots ,{x}_{2M}]\\
{X}_f &= [{y}_1,\cdots , {y}_M,y_1+\delta y_1, \cdots , y_M + \delta y_M]\\
& = [{y}_1,\cdots ,{y}_{2M}]
\end{aligned}
\end{eqnarray}
is created with $2(M+1)$ data points. Here $x_{M+i} = x_i + \delta x_i$ and $y_{M+i} = y_i + \delta y_i$.

\subsection{Robust Optimization Formulation}

Let $X_p=[{x}_1,\cdots ,{x}_{2M}]$ and $X_f=[y_1,\cdots y_{2M}]$ be the enriched data set, as given in (\ref{enriched_dataset}). However, all the data points $(x_i,y_i)$ are not obtained from the system but has some error due to the artificial data points. The errors in the artificial data points can be thought of as measurement noise and the uncertainty acts as an adversary which tries to maximize the residual. Hence to obtain the Koopman operator $\bf K$ for uncertain data, a robust optimization problem can be formulated as the following $\min-\max$ optimization problem. 

\begin{equation}\label{edmd_robust}
\min\limits_{\bf K}\max_{\delta\in \Delta}\parallel {\bf G}_\delta{\bf K}-{\bf A}_\delta\parallel_F=:\min\limits_{\bf K}\max_{\delta\in \Delta} {\cal J}({\bf K}, {\bf G}_\delta,{\bf A}_\delta)
\end{equation}
where
\begin{eqnarray}
&&{\bf G}_\delta=\frac{1}{2M}\sum_{i=1}^{2M} \boldsymbol{\Psi}({ x}_i)^\top \boldsymbol{\Psi}({x}_i)\nonumber\\
&&{\bf A}_\delta=\frac{1}{2M}\sum_{i=1}^{2M} \boldsymbol{\Psi}({ x}_i)^\top \boldsymbol{\Psi}({ y}_{i}),
\end{eqnarray}
with ${\bf K},{\bf G}_\delta,{\bf A}_\delta\in\mathbb{C}^{K\times K}$.

The robust optimization problem (\ref{edmd_robust}), is in general non-convex because the cost $\cal J$ may not be a convex function of $\delta$. 

\begin{proposition}
The optimization problem (\ref{edmd_robust}) can be approximated as
\begin{equation}\label{edmd_robust_convex}
\min\limits_{\bf K}\max_{\delta{\bf G},\delta{\bf A}\in  {\cal U}}\parallel ({\bf G}+\delta {\bf G}){\bf K}-({\bf A}+\delta {\bf A})\parallel_F
\end{equation}
where $\cal U$ is a compact set in $\mathbb{R}^{K\times K}$.
\end{proposition}
\begin{proof}
From Taylor series expansion we have, ${\bf \Psi}(x_i+\delta x_i) = {\bf \Psi}(x_i) +  {\bf \Psi}'(x_i) \delta x_i+h.o.t.$, where ${\bf \Psi}'(x_i)$ is the first derivative of ${\bf \Psi}(x)$ at $x_i$. Hence,
\begin{eqnarray*}
{\bf G}_{\delta} &\approx& {\bf G} + \frac{1}{2M}\sum_{i=1}^{2M}{\bf \Psi}^\top (x_i)\delta x_i {\bf \Psi}'(x_i)\\
&=& {\bf G} +\delta {\bf G}
\end{eqnarray*}
where $\delta {\bf G} = \frac{1}{2M}\sum_{i=1}^{2M}{\bf \Psi}^\top (x_i)\delta x_i {\bf \Psi}'(x_i)$.

Moreover,
\begin{eqnarray*}
&&\parallel \delta {\bf G} \parallel_F = \parallel \frac{1}{2M}\sum_{i=1}^{2M}{\bf \Psi}^\top (x_i)\delta x_i {\bf \Psi}'(x_i) \parallel_F\\
&\leq& \frac{1}{2M}\sum_{i = 1}^{2M}\parallel {\bf \Psi}^\top (x_i)\delta x_i {\bf \Psi}'(x_i)\parallel_F \\
&\leq&\frac{1}{2M}\sum_{i=1}^{2M}\parallel {\bf \Psi}^\top (x_i)\parallel_F \cdot \parallel\delta x_i \parallel_F \cdot \parallel {\bf \Psi}'(x_i)\parallel_F
\end{eqnarray*}
Hence, $\delta {\bf G}$ belongs to a compact set ${\cal U}_1$. Similarly, one can show ${\bf A}_\delta \approx {\bf A} + \delta {\bf A}$ and $\delta {\bf A}$ belongs to a compact set ${\cal U}_2$. Letting ${\cal U}={\cal U}_1\cup{\cal U}_2$, proves the proposition.
\end{proof}
The above proposition allows us to compute the Koopman operator $\bf K$ as a solution of a robust optimization problem (\ref{edmd_robust_convex}). The optimization problem (\ref{edmd_robust_convex}) has interesting connections with optimization problem involving regularization. In particular, one has the following theorem.

\begin{theorem}
The optimization problem 
\begin{equation}
\min\limits_{\bf K}\max_{\delta{\bf G},\delta{\bf A}\in  {\cal U}}\parallel ({\bf G}+\delta {\bf G}){\bf K}-({\bf A}+\delta {\bf A})\parallel_F
\end{equation}
is equivalent to the following optimization problem
\begin{eqnarray}\label{rob_eqv}
\min\limits_{\bf K}\parallel {\bf G}{\bf K}-{\bf A}\parallel_F+\lambda \parallel {\bf K}\parallel_F\label{regular}
\end{eqnarray}
\end{theorem} 
\begin{proof}
The min-max optimization problem is 
\begin{eqnarray}
{\cal J} = \min_{\bf K}\max_{\delta{\bf G},\delta{\bf A}\in  {\cal U}}\parallel ({\bf G}+\delta {\bf G}){\bf K}-({\bf A}+\delta {\bf A})\parallel_F
\end{eqnarray}
Fix ${\bf K}\in\mathbb{R}^{K\times K}$ and let
\begin{eqnarray*}
r = \max_{\delta{\bf G},\delta{\bf A}\in  {\cal U}}\parallel ({\bf G}+\delta {\bf G}){\bf K}-({\bf A}+\delta {\bf A})\parallel_F
\end{eqnarray*}
be the worst-case residual.
Let $\max\{\parallel u \parallel_F | u \in {\cal U}\}\leq \lambda$. Then,
\begin{eqnarray}\nonumber\label{opt_reg_leq}
&&r = \max_{\delta{\bf G},\delta{\bf A}\in  {\cal U}} \parallel ({\bf G}+\delta {\bf G}){\bf K}-({\bf A}+\delta {\bf A})\parallel_F\\ \nonumber
&& \leq \max_{\delta{\bf G},\delta{\bf A}\in  {\cal U}} \parallel {\bf G} {\bf K} - {\bf A} \parallel_F + \parallel \delta {\bf G}{\bf K} - \delta {\bf A} \parallel_F\\
&& \leq \parallel {\bf G} {\bf K} - {\bf A} \parallel_F + \lambda \parallel {\bf K}\parallel_F + \lambda \parallel{\bf I}\parallel_F \\ 
&&\leq \parallel {\bf G} {\bf K} - {\bf A} \parallel_F + \lambda \parallel {\bf K}\parallel_F + \lambda K\end{eqnarray}
Again, let $\begin{pmatrix}
\delta {\bf G}^\star &, & \delta {\bf A}_i^\star
\end{pmatrix} = \frac{\lambda u}{\sqrt \parallel {\bf K}_i\parallel_2 + 1}\begin{pmatrix}
{\bf K}_i^\top &,& & -1
\end{pmatrix}$
where ${\bf A}_i$ and ${\bf K}_i$ are the $i^{th}$ columns of $\bf A$ and $\bf K$ respectively and 
\begin{eqnarray*}
u = \begin{cases}
     \frac{{\bf G}{\bf K}_i-{\bf A}_i}{\parallel {\bf G}{\bf K}_i-{\bf A}_i \parallel_2},\textnormal{ if } {\bf G}{\bf K}_i\neq{\bf A}_i\\
     \textnormal{any unit norm vector otherwise.}
    \end{cases}
\end{eqnarray*}
Let 
\begin{eqnarray*}
r_i = \max_{\delta{\bf G},\delta{\bf A}\in  {\cal U}}\parallel ({\bf G}+\delta {\bf G}){\bf K}_i-({\bf A}_i+\delta {\bf A}_i)\parallel_F.
\end{eqnarray*}
Hence,
\begin{eqnarray}\nonumber
&&r_i\geq \parallel ({\bf G}{\bf K}_i-{\bf A}_i)+(\delta {\bf G}^\star {\bf K}_i -\delta {\bf A}_i^\star)\parallel_F\\ \nonumber
&& = \parallel ({\bf G}{\bf K}_i-{\bf A}_i) + \frac{\lambda({\bf G}{\bf K}_i-{\bf A}_i)}{\parallel {\bf G}{\bf K}_i-{\bf A}_i \parallel_2}({\bf K}_i^\top {\bf K}_i+1)\parallel_F\\
&&= \parallel {\bf G}{\bf K}_i-{\bf A}_i\parallel_F + \lambda \parallel {\bf K}_i\parallel_F + \lambda
\end{eqnarray}
Hence, 
\begin{eqnarray}\label{opt_reg_geq}
r = \sum_{i=1}^K r_i \geq \parallel {\bf G} {\bf K} - {\bf A} \parallel_F + \lambda \parallel {\bf K}\parallel_F + \lambda K.
\end{eqnarray}

Finally, from (\ref{opt_reg_leq}) and (\ref{opt_reg_geq}) and minimizing over $\bf K$ we get
\begin{eqnarray*}
\min\limits_{\bf K}\max_{\delta{\bf G},\delta{\bf A}\in  {\cal U}}\parallel ({\bf G}+\delta {\bf G}){\bf K}-({\bf A}+\delta {\bf A})\parallel_F
\end{eqnarray*}
is equivalent to 
\begin{eqnarray*}
\min\limits_{\bf K}\parallel {\bf G} {\bf K} - {\bf A} \parallel_F + \lambda \parallel {\bf K}\parallel_F + \lambda K.
\end{eqnarray*}
Moreover, since $\lambda$ and $K$ are positive constants,
\begin{eqnarray*}
\min\limits_{\bf K}\max_{\delta{\bf G},\delta{\bf A}\in  {\cal U}}\parallel ({\bf G}+\delta {\bf G}){\bf K}-({\bf A}+\delta {\bf A})\parallel_F
\end{eqnarray*}
is equivalent to 
\begin{eqnarray*}
\min\limits_{\bf K}\parallel {\bf G} {\bf K} - {\bf A} \parallel_F + \lambda \parallel {\bf K}\parallel_F.
\end{eqnarray*}
\end{proof}

The above theorem allows computation of the approximate Koopman operator as a solution of an optimization problem with a regularization term. In particular, the approximate Koopman operator can be obtained as a solution of the following optimization problem
\begin{eqnarray}\label{robust_opt}
\parallel {\bf G} {\bf K} - {\bf A} \parallel_F + \lambda \parallel {\bf K}\parallel_F.
\end{eqnarray}
where
\begin{eqnarray}
\begin{aligned}
&{\bf G}=\frac{1}{2M}\sum_{i=1}^{2M} \boldsymbol{\Psi}({ x}_i)^\top \boldsymbol{\Psi}({x}_i)\\
&{\bf A}=\frac{1}{2M}\sum_{i=1}^{2M} \boldsymbol{\Psi}({ x}_i)^\top \boldsymbol{\Psi}({ y}_{i}).
\end{aligned}
\end{eqnarray}
\begin{algorithm}
\caption{Computation of Koopman Operator from Limited Data}
\begin{enumerate}
\item{To the existing data set $D = [x_1,\cdots, x_M]$ add new data points $\tilde{x}_i = x_i+\delta x_i$, where $\parallel \delta x_i\parallel \leq c$.}
\item{Form the enriched data set $\bar{D} = [x_1,\cdots, x_M, \tilde{x}_1,\cdots, \tilde{x}_M]$.}
\item{Form the sets $X_p = [x_1,\cdots, x_{M-1}, \tilde{x}_1,\cdots, \tilde{x}_{M-1}]$ and $X_f = [x_2,\cdots, x_M, \tilde{x}_2,\cdots, \tilde{x}_M]$.}
\item{Fix the dictionary functions ${\bf \Psi} = [\psi_1,\cdots , \psi_K]$.}
\item{Solve the optimization problem to obtain the approximate Koopman operator $\bf K$
\begin{eqnarray*}
\parallel {\bf G} {\bf K} - {\bf A} \parallel_F + \lambda \parallel {\bf K}\parallel_F.
\end{eqnarray*}
where
\begin{eqnarray*}
&&{\bf G}=\frac{1}{2M}\sum_{i=1}^{2M} \boldsymbol{\Psi}({ x}_i)^\top \boldsymbol{\Psi}({x}_i)\nonumber\\
&&{\bf A}=\frac{1}{2M}\sum_{i=1}^{2M} \boldsymbol{\Psi}({ x}_i)^\top \boldsymbol{\Psi}({ y}_{i}).
\end{eqnarray*}
}
\end{enumerate}
\label{algo}
\end{algorithm}

\begin{remark}
Algorithm (\ref{algo}) describes the procedure of adding $M$ extra data points, but it should be noted that one can add more or less number of artificial data points.
\end{remark}

\section{Design of Robust Predictor}\label{section_predictor}
The Koopman operator generates a linear system in a higher dimensional space, even if the underlying system is linear. The linearity of the operator enables the design of linear predictors for nonlinear systems. The following is presented briefly for the self-containment of the paper and for details the readers are referred to \cite{korda_mezic_predictor}. Let $\{x_0,\ldots,x_M\}$ be the training data-set and $\bf K$  be the finite-dimensional approximation of the transfer Koopman operator obtained using algorithm ${\bf 1}$. Let $\bar x_0$ be the initial condition from which the future is to be predicted. The initial condition from state space is mapped to the feature space using the same choice of basis function used in the robust approximation of Koopman operator i.e., \[\bar x_0\implies {\bf \Psi}(\bar x_0)^\top=: {\bf z}\in \mathbb{R}^K.\] This initial condition is propagated using Koopman operator as \[{\bf z}_n={\bf K}^n{\bf z}.\]
The predicted trajectory in the state space is then obtained as 
\[\bar x_n=C {\bf z}_n\]
where matrix $C$ is obtained as the solution of the following least squares problem
\begin{eqnarray}\label{C_pred}
\min_C\sum_{i = 1}^M \parallel x_i - C \boldsymbol \Psi (x_i)\parallel_2^2
\end{eqnarray}

\section{Simulations}\label{section_simulation}

In this section, we demonstrate the efficiency of the proposed algorithm on three different dynamical systems. In particular, we construct the Koopman operator for a linear system, a non-linear system and a system governed by a Partial Differential Equation (PDE). 
\subsection{Network of Coupled Oscillators}
Consider a network of coupled linear oscillators given by
\begin{eqnarray}\label{coup_osc}
\ddot{\theta}_k &=& -\mathcal{L}_k\theta - d\dot{\theta}_k, \quad k = 1,\cdots , N
\end{eqnarray}
where $\theta_k$ is the angular position of the $k^{th}$ oscillator, $N$ is the number of oscillators, $\mathcal{L}_k$ is the $k^{th}$ row of the Laplacian $\mathcal{L}$ and $d$ is the damping coefficient. The Laplacian $\cal L$ is chosen such that the network is a ring network with 20 oscillators (Fig. \ref{fig_oscillator}).

\begin{figure}[htp!]
\centering
\includegraphics[scale=.6]{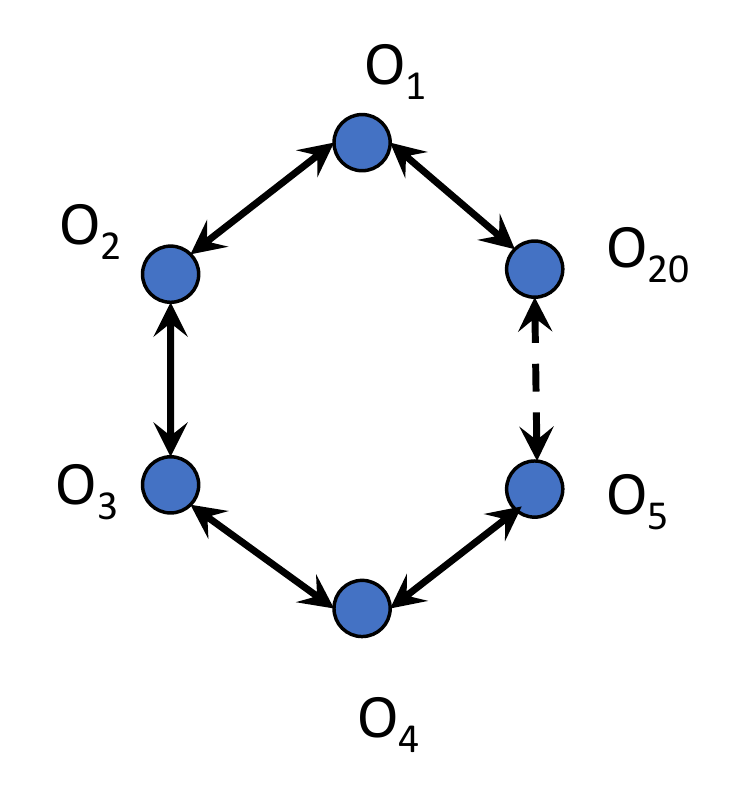}
\caption{Ring network of 20 linear oscillators.}\label{fig_oscillator}
\end{figure}

In these sets of simulations, the damping coefficient $d$ has been assumed the same for all the oscillators and is set equal to $0.4$. Data for all the states were collected for 100-time steps, with sampling time $\delta t = 0.01$ seconds and since the system is linear, linear basis functions were used for computation of the Koopman operator. 
\begin{figure}[htp!]
\centering
\subfigure[]{\includegraphics[scale=.21]{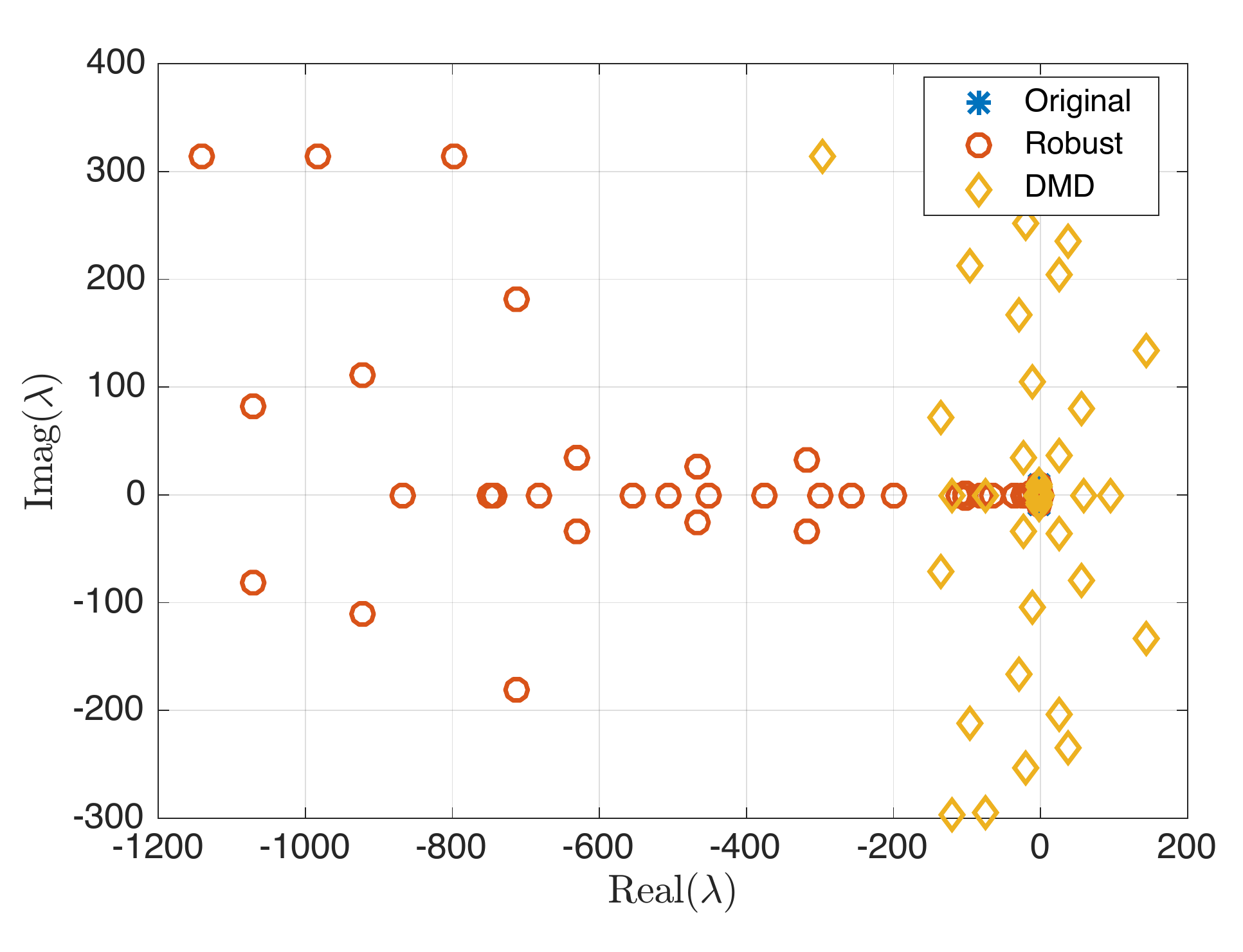}}
\subfigure[]{\includegraphics[scale=.21]{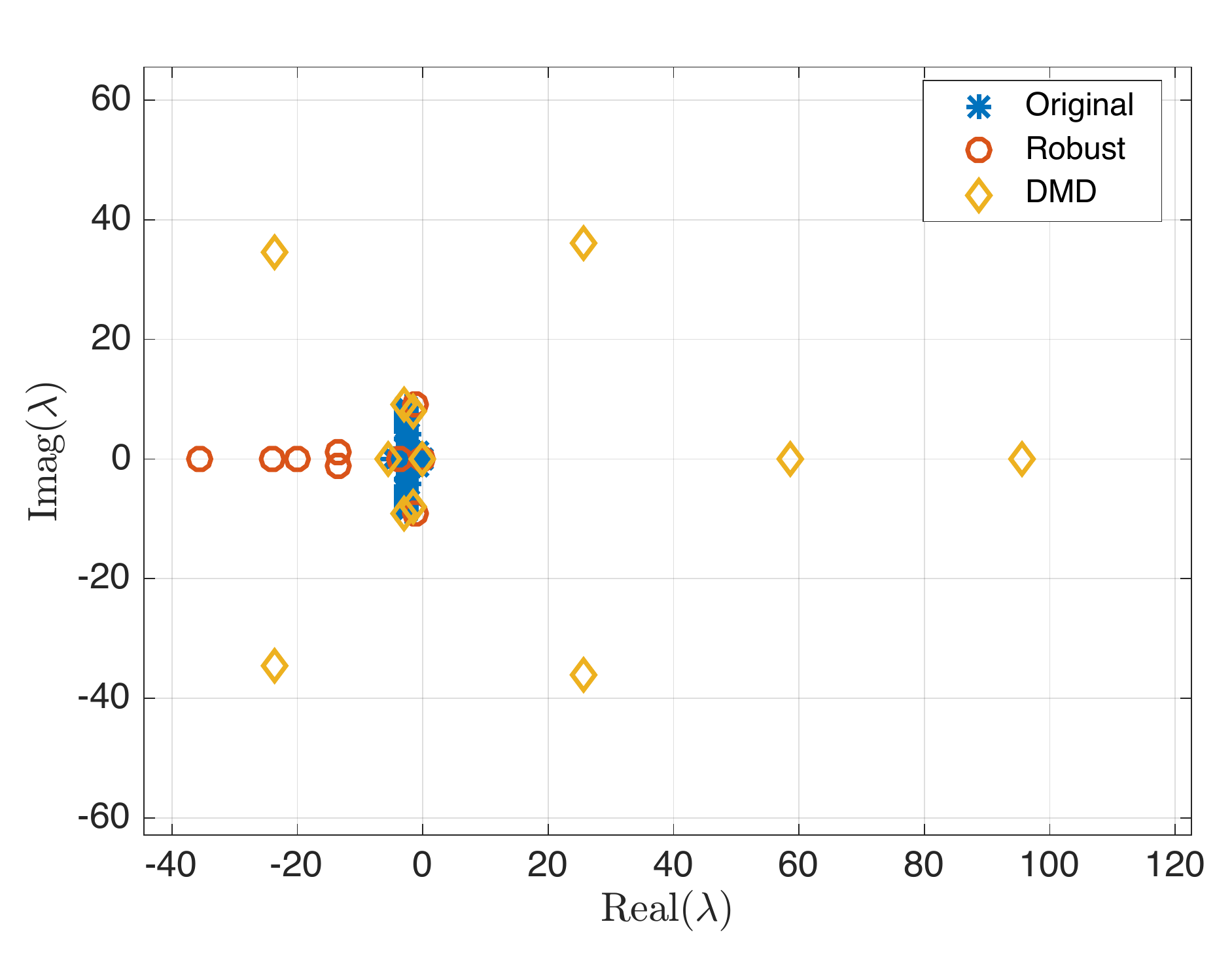}}
\caption{(a) Eigenvalue obtained using normal DMD on original training data and Robust DMD on enriched data set. (b) Dominant eigenvalues.}\label{eig_osc}
\end{figure}
The first 15-time steps data was used for training the Koopman operators. Normal DMD on the 15 data points yields positive eigenvalues with a significant real part, as shown in Fig. \ref{eig_osc}. For the Robust identification of Koopman operator, the original data set was enriched by adding 30 artificial data points and Robust DMD formulation (algorithm \ref{algo}) yields a much better approximation of the eigenvalues for the original system. The eigenvalues obtained using normal DMD and Robust DMD are shown in Fig. \ref{eig_osc}, where in Fig. \ref{eig_osc}a the complete spectrum is plotted and in Fig. \ref{eig_osc}b the dominant eigenvalues are shown. 

\begin{figure}[htp!]
\centering
\subfigure[]{\includegraphics[scale=.215]{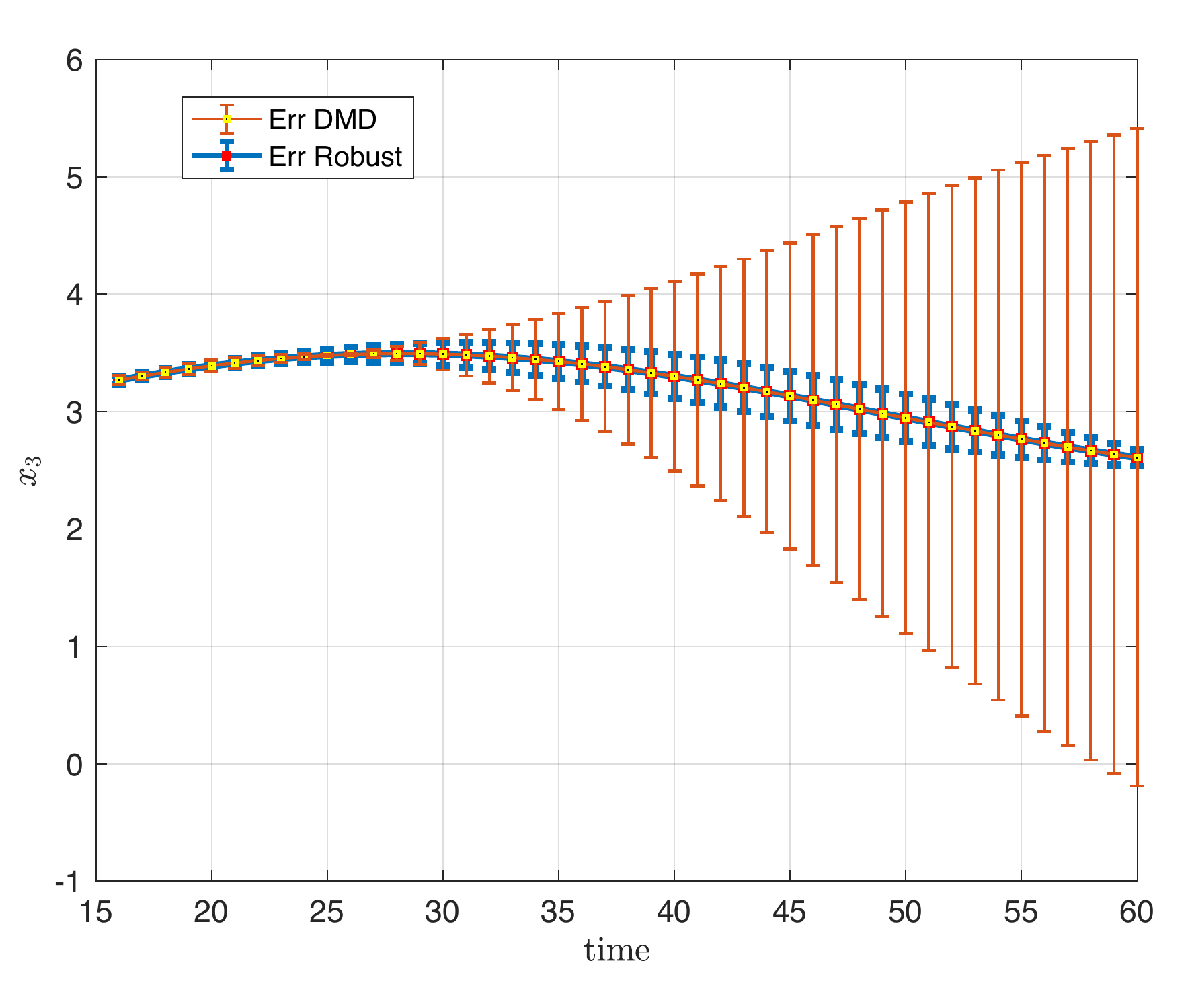}}
\subfigure[]{\includegraphics[scale=.215]{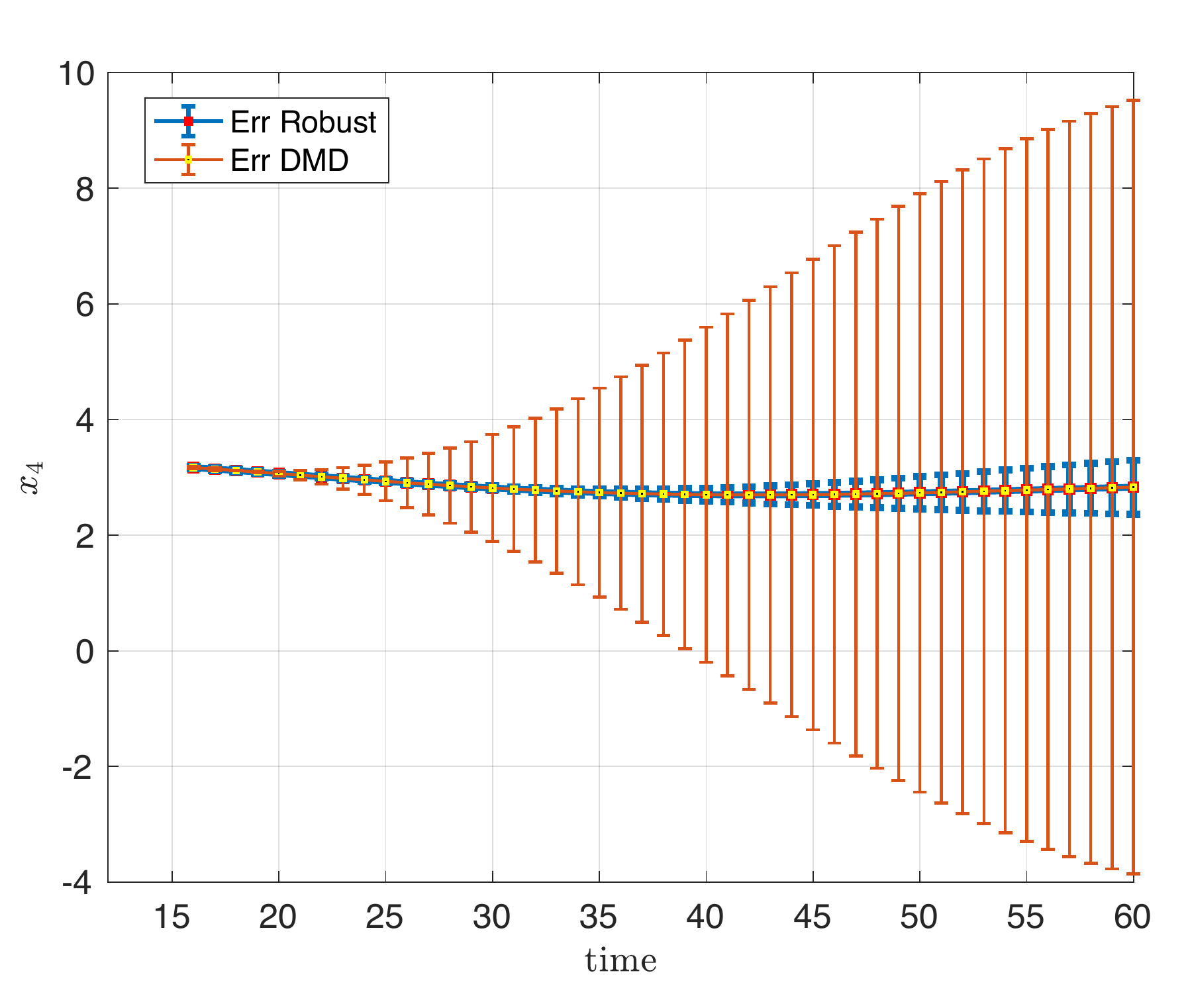}}
\caption{(a) Errors in the prediction of position of oscillator 3. (b) Errors in the prediction of position of oscillator 4.}\label{err_osc}
\end{figure}
As mentioned earlier, data were obtained for 100 times steps and the first 15 time steps were used for training the Koopman operator. Koopman operators thus obtained was used to predict the next 45 time steps and was used to compare the error. The errors in the prediction of the positions of oscillators 3 and 4, using both normal DMD and Robust DMD, are shown in Fig. \ref{err_osc}a and Fig. \ref{err_osc}b respectively. It can be observed that Robust DMD formulation generates much smaller error compared to normal DMD. In fact, this was expected, since Robust DMD with enriched data-set approximates the eigenspectrum much better compared to normal DMD.

\subsection{Stuart-Landau Equation}
The nonlinear Stuart-Landau equation on a complex function $z(t) = r(t)\exp (i\theta(t))$ is given by 
\begin{eqnarray}\label{stu_lan}
\dot{z} = (\mu + i\gamma)z - (1+ i\beta)|z|^2z ,
\end{eqnarray}
where $i$ is the imaginary unit. The solution of (\ref{stu_lan}) evolves on the limit cycle $|z| = \sqrt{\mu}$. Hence, the continuous time eigenvalues lie on the imaginary axis. The discretized version of (\ref{stu_lan}) is

{\small
\begin{eqnarray}\label{stu_lan_dis}
\begin{pmatrix}
r_{t+1}\\
\theta_{t+1}
\end{pmatrix} = \begin{pmatrix}
r_t + (\mu r_t -r_t^3)\delta t\\
\theta_t + (\gamma - \beta r_t^2)\delta t
\end{pmatrix}
\end{eqnarray}
}

The set of dictionary functions were chosen as
\begin{eqnarray}\label{stuart_dic}
{\bf \Psi}(\theta_t) = \begin{pmatrix}
e^{-10i\theta_t} & e^{-9i\theta_t} & \cdots & e^{9i\theta_t} & e^{10i\theta_t}
\end{pmatrix}
\end{eqnarray}
 and data was collected for 150 time steps, with $\delta t = 0.01$ and initial condition $(1,\pi)$. 

\begin{figure}[htp!]
\centering
\subfigure[]{\includegraphics[scale=.2]{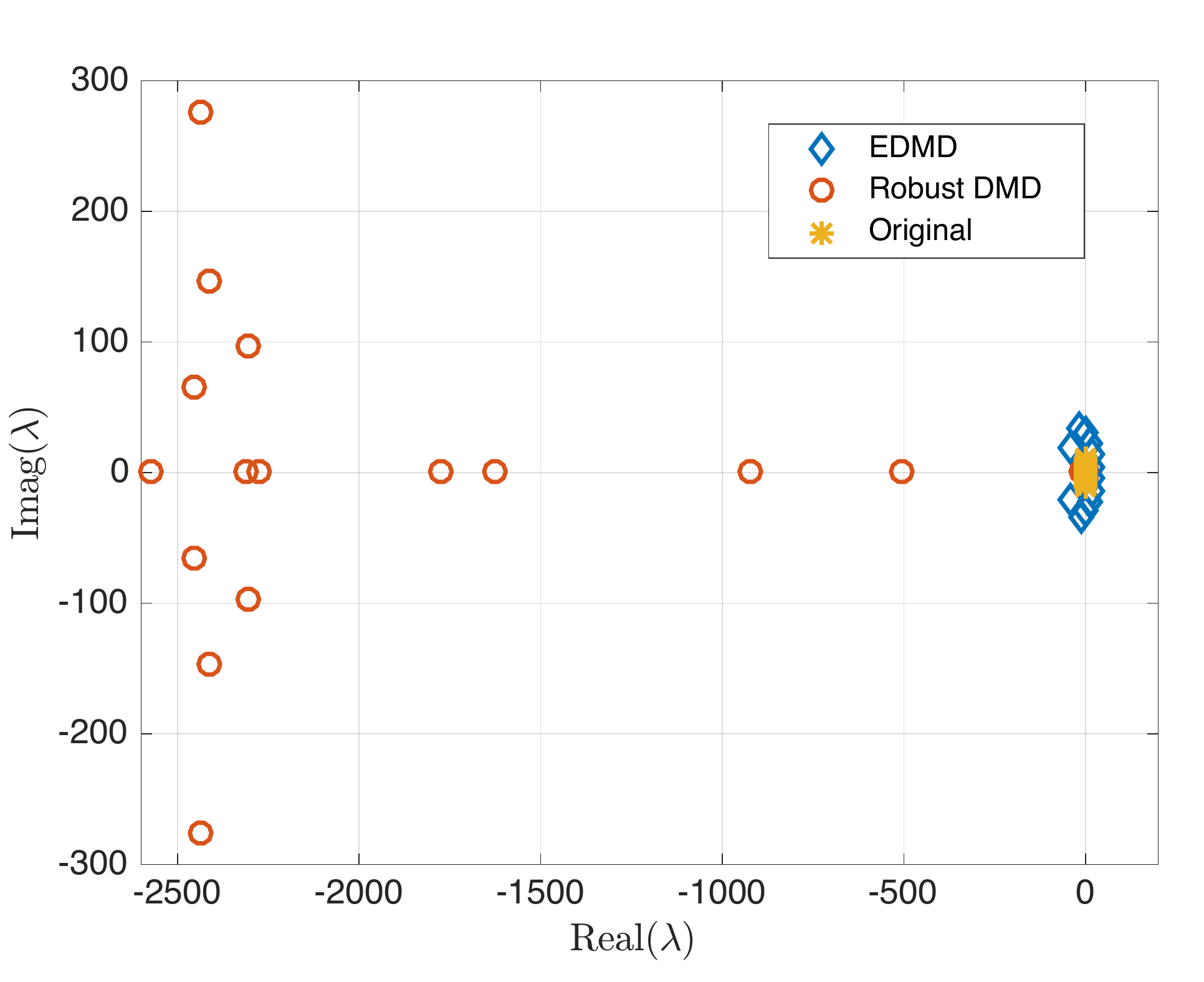}}
\subfigure[]{\includegraphics[scale=.2]{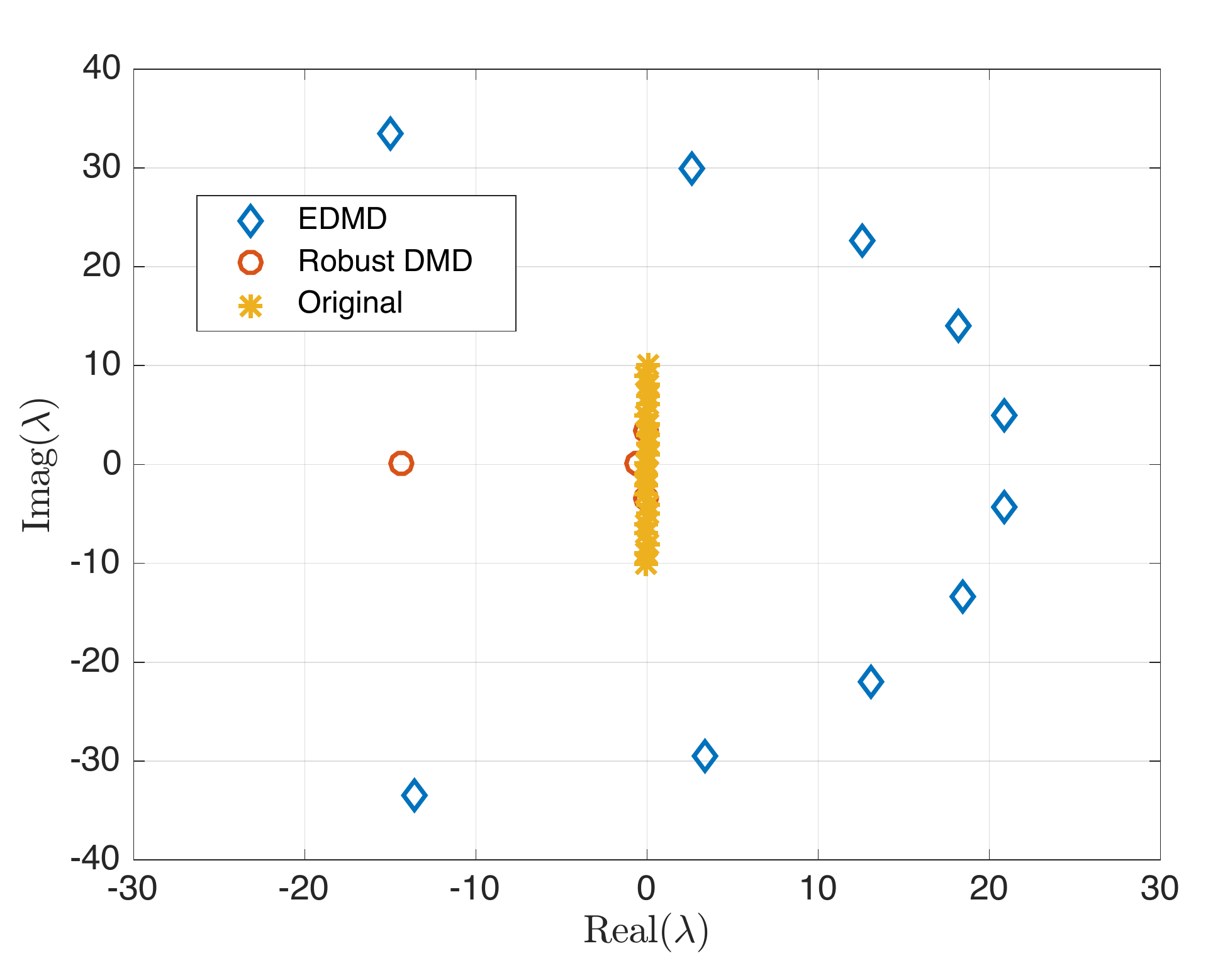}}
\caption{(a) Eigenvalues obtained with EDMD on original data and Robust EDMD on enriched data set. (b) Dominant eigenvalues.}\label{eig_nonlinear}
\end{figure}

The first 30-time steps data were used as the training data for training the Koopman operator. An extra 30 artificial points were added to the obtained data set to form the enriched data set and this enriched data set was used to compute the eigenspectrum of the Koopman operator using Robust EDMD algorithm. The eigenvalues obtained using the dictionary functions given in (\ref{stuart_dic}), with normal EDMD and Robust EDMD with enriched data set is shown in Fig. \ref{eig_nonlinear}a. Fig. \ref{eig_nonlinear}b shows the dominant eigenvalues and it can be observed that Robust EDMD provides a better approximation of the original eigenspectrum. In particular, normal EDMD generates unstable eigenvalues.

\begin{figure}[htp!]
\centering
\includegraphics[scale=.35]{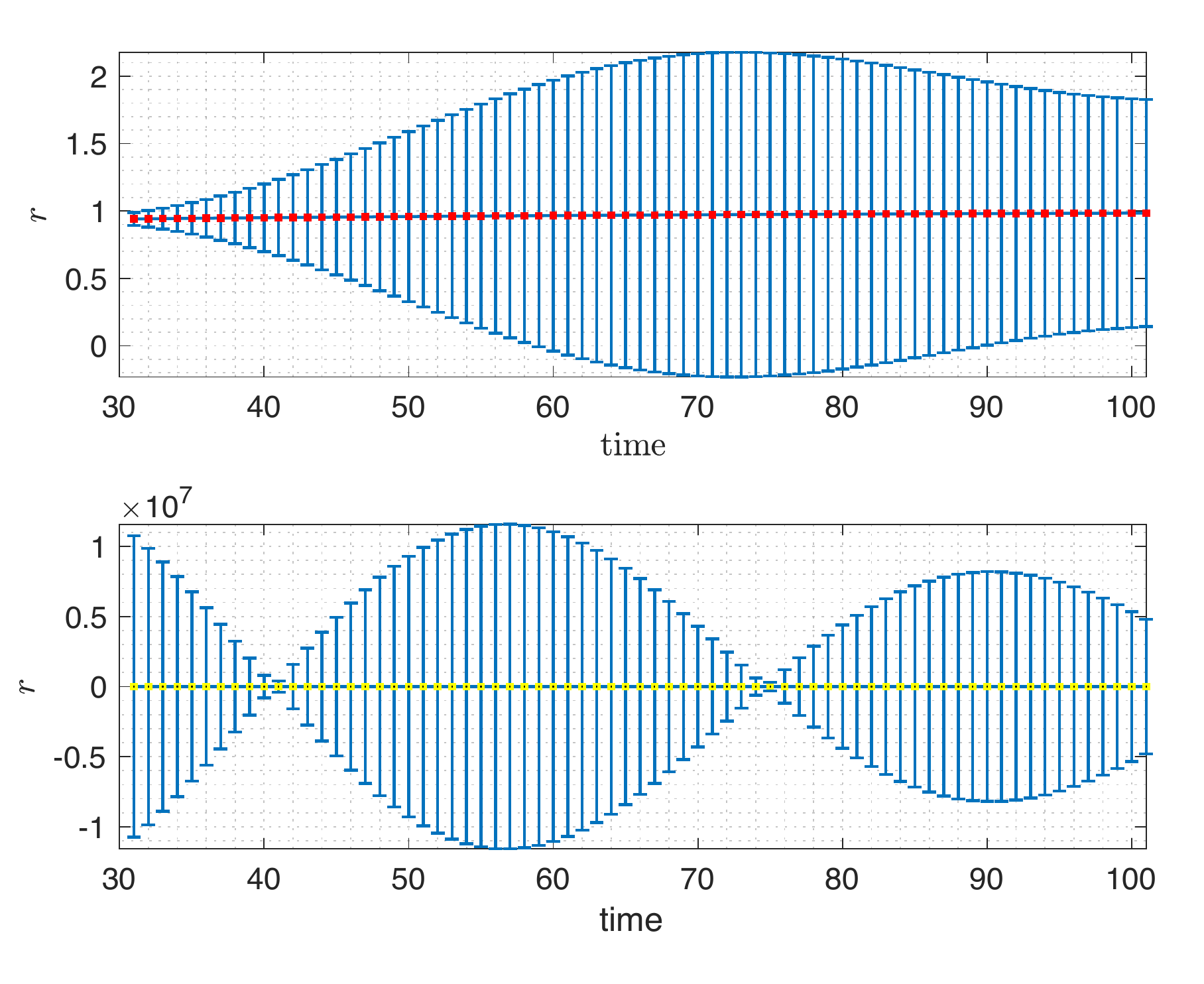}
\caption{Comparison of errors in prediction of $r$ using Robust EDMD and normal EDMD. The top figure shows the prediction error using Robust EDMD and the lower plot shows prediction error using normal EDMD.}\label{prediction_r}
\end{figure}

\begin{figure}[htp!]
\centering
\includegraphics[scale=.35]{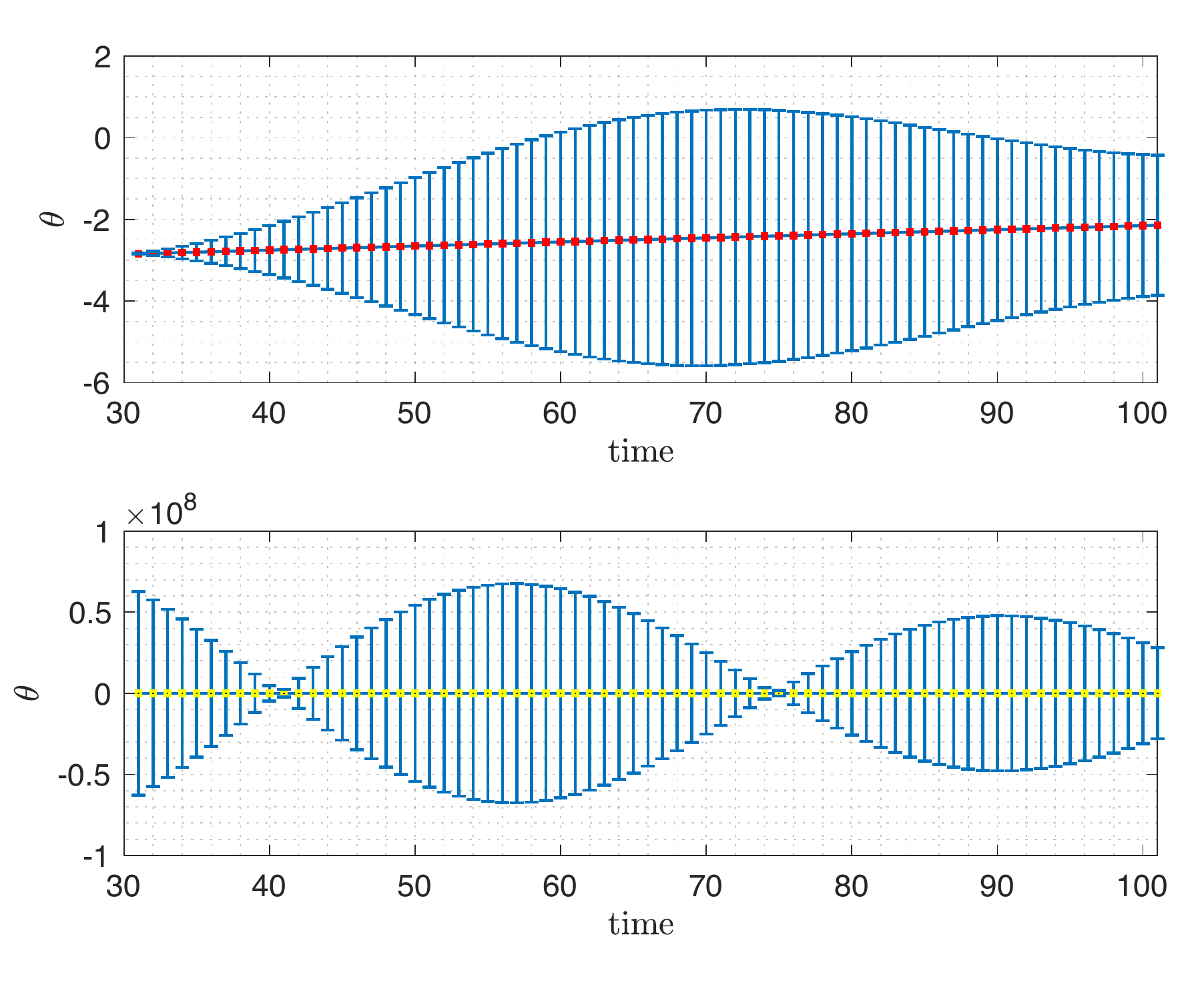}
\caption{Comparison of errors in prediction of $\theta$ using Robust EDMD and normal EDMD. The top figure shows the prediction error using Robust EDMD and the lower plot shows prediction error using normal EDMD.}\label{prediction_theta}
\end{figure}

Further, using the Koopman operators obtained using both normal EDMD and Robust EDMD, future values of both $r$ and $\theta$ was predicted for the next 70 time steps. The errors in the prediction of $r$ and $\theta$ are shown in Fig. \ref{prediction_r} and Fig. \ref{prediction_theta} respectively. In all the error plots, the errors are plotted against the actual values of $r$ and $\theta$ and it can be observed that the errors in prediction for both $r$ and $\theta$ with Robust EDMD are significantly smaller than the prediction errors using normal EDMD.

\subsection{Burger-Equation}
The third example considered in this paper is the Burger equation. Burger equation is a successful but simplified partial differential equation which describes the motion of viscous compressible fluids. The equation is of the form 
\[\partial_t u(x,t)+u\partial_x u=k\partial_x^2u\]
where $u$ is the speed of the gas, $k$ is the kinematic viscosity, $x$ is the spatial coordinate and $t$ is time. 

In the simulation, choosing $k=0.01$, we approximated the PDE solution using the Finite Difference method \cite{KUTLUAY1999251} with the initial condition $u(x,0)=sin(2\pi x)$ and Dirichet boundary condition $u(0,t)=u(1,t)=0$. Given the spatial and temporal ranges, $x\in[0,1],\;t\in[0,1]$, the discretizaion steps are chosen as $\Delta t=0.02$ and $\Delta x=1\times10^{-2}$. With the above set of conditions, the flow $u$ is shown in Fig. \ref{burger_flow}.
\begin{figure}[htp!]
\centering
\includegraphics[scale=.275]{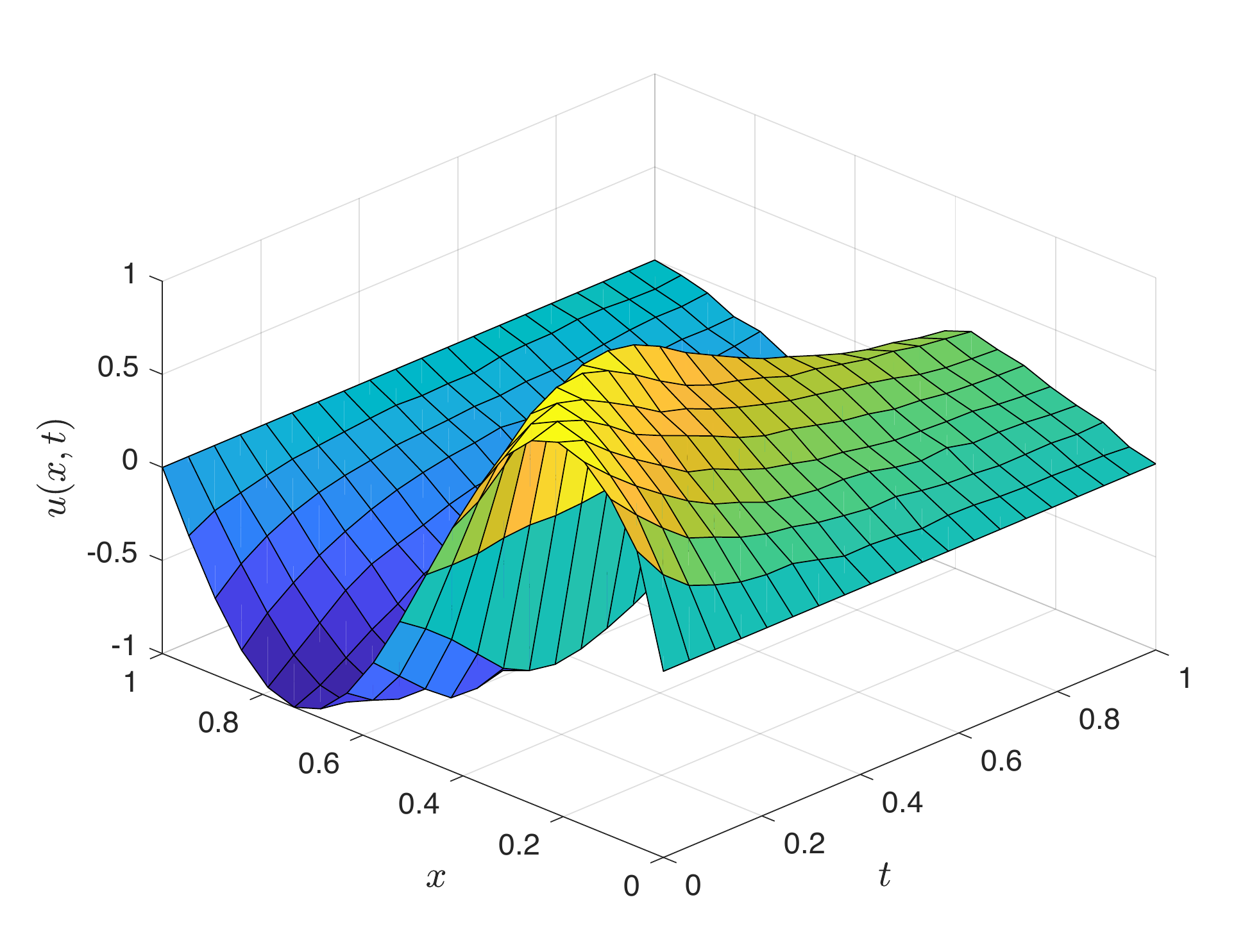}
\caption{Flow field of Burger equation.}\label{burger_flow}
\end{figure}

Since the space discretization was chosen as $\Delta x = 1\times 10^{-2}$, there are 100 state variables. For computing the Koopman operator, 8-time steps data were used. 40 extra data points were added to enrich the data set and the Robust Koopman operator was computed using the enriched data set. Koopman operator using normal DMD was also computed for comparing the errors in prediction. The errors in the prediction of 35 future time steps for $x_{40}$ and $x_{100}$ is shown in Fig. \ref{burger_error}(a) and Fig.. \ref{burger_error}(b) respectively. It can be seen that the error in prediction using Robust Koopman operator from the enriched data set is much smaller as compared to the normal DMD.

\begin{figure}[htp!]
\centering
\subfigure[]{\includegraphics[scale=.205]{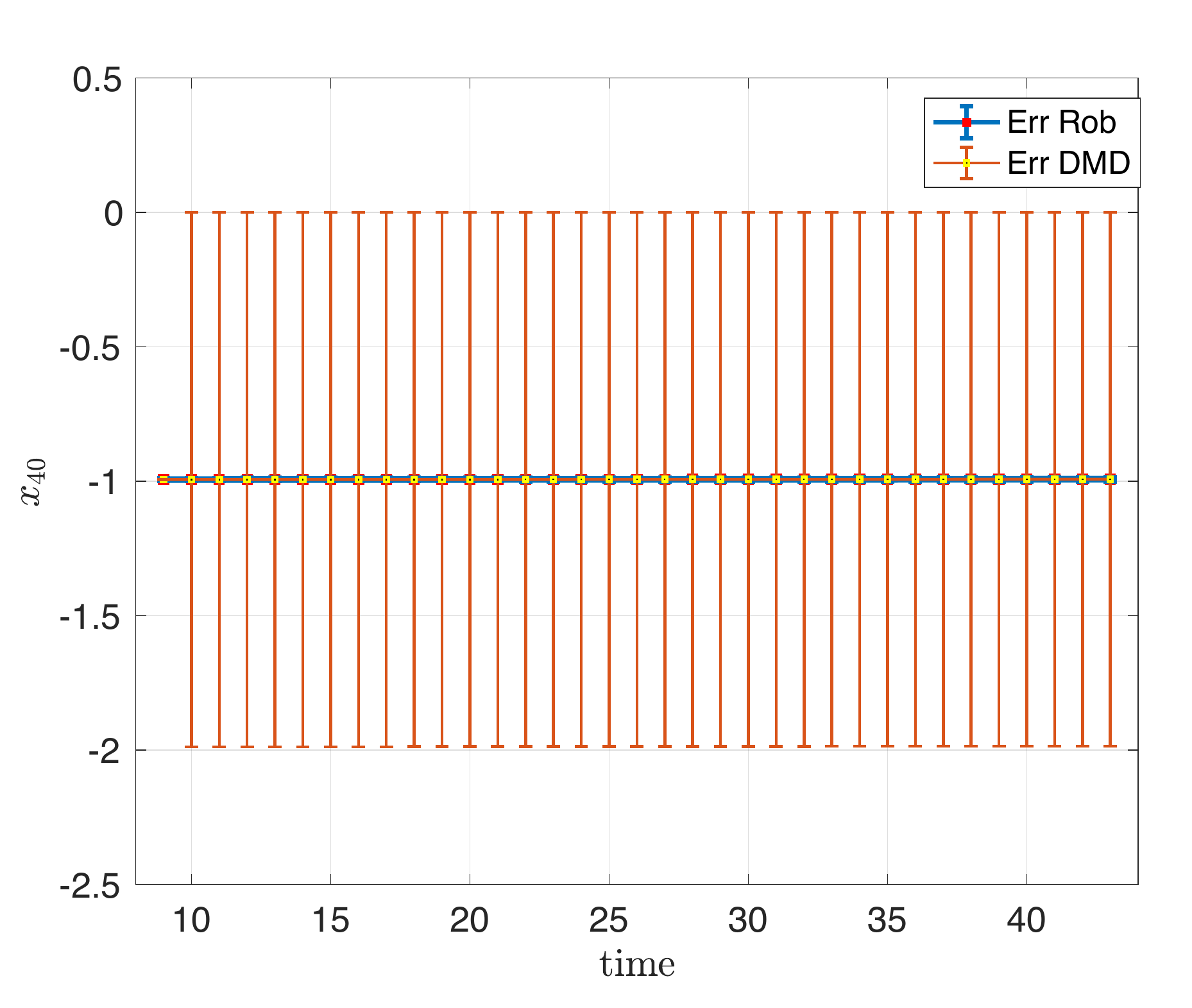}}
\subfigure[]{\includegraphics[scale=.2]{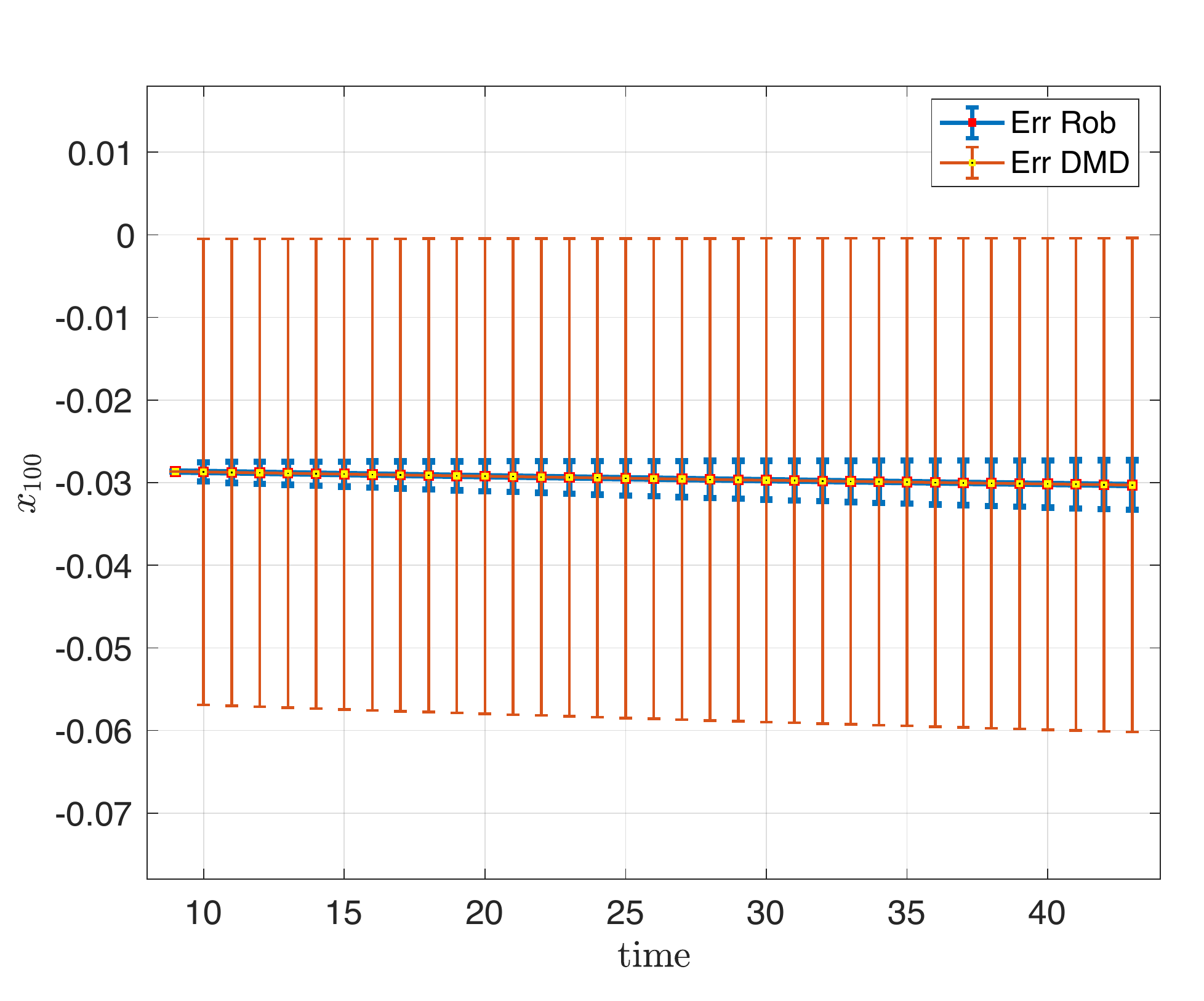}}
\caption{(a) Errors in prediction of $x_{40}$. (b) Errors in prediction of $x_{100}$.}\label{burger_error}
\end{figure}

We further used different training size data for computing the Koopman operator and compared the mean square error in prediction of all the states. In particular, we used both Robust DMD approach and normal DMD to predict 35-time steps from $t=100$, with 7 different training size data, namely 5, 10, 15, 20, 25, 30 and 35-time steps. For each of the training size data, we appended the data set with artificial data points so that there are 40 data points in total. The mean square errors in the prediction of the states are shown in Fig. \ref{burger_mse}. 

\begin{figure}[htp!]
\centering
\subfigure[]{\includegraphics[scale=.45]{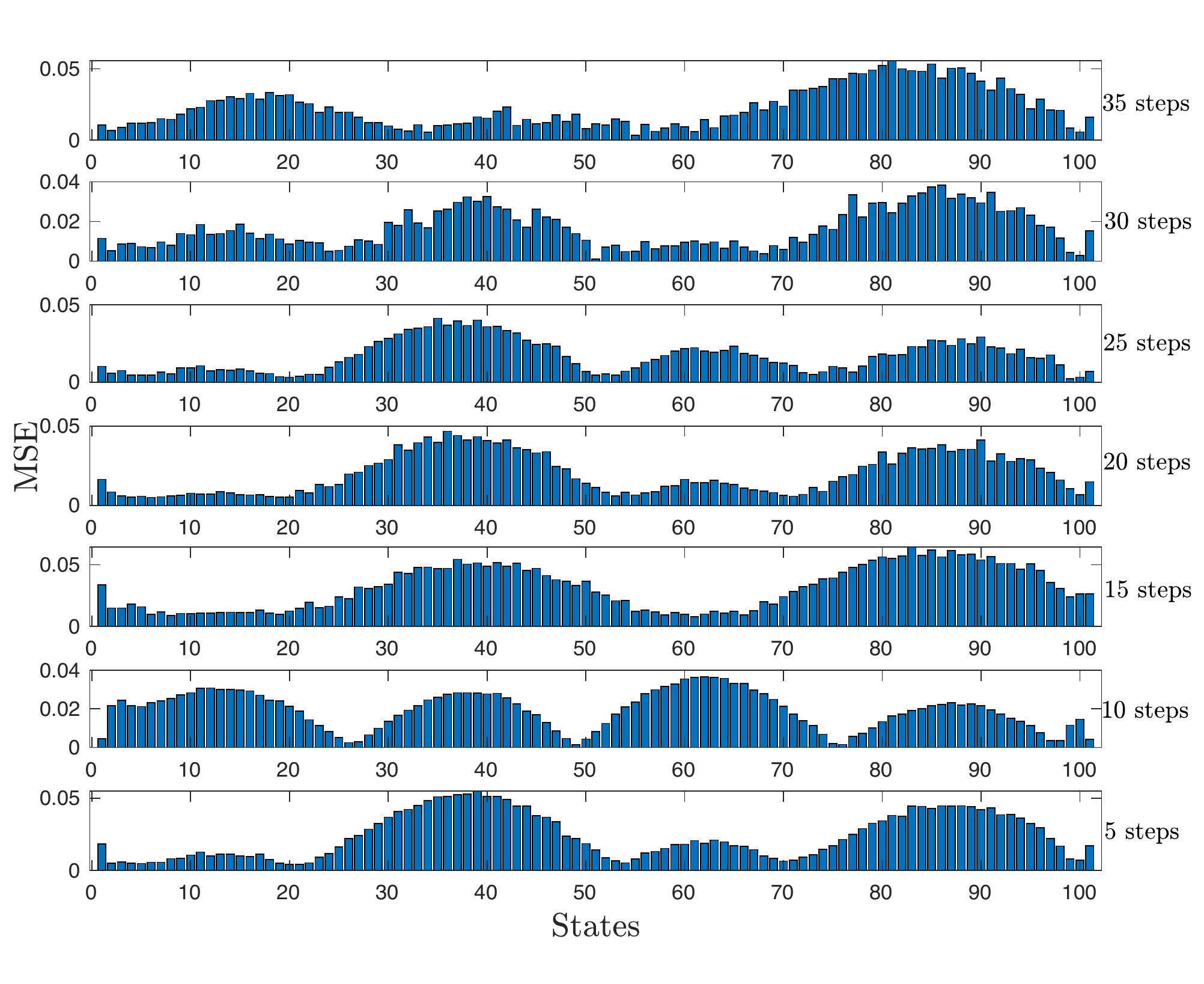}}
\subfigure[]{\includegraphics[scale=.45]{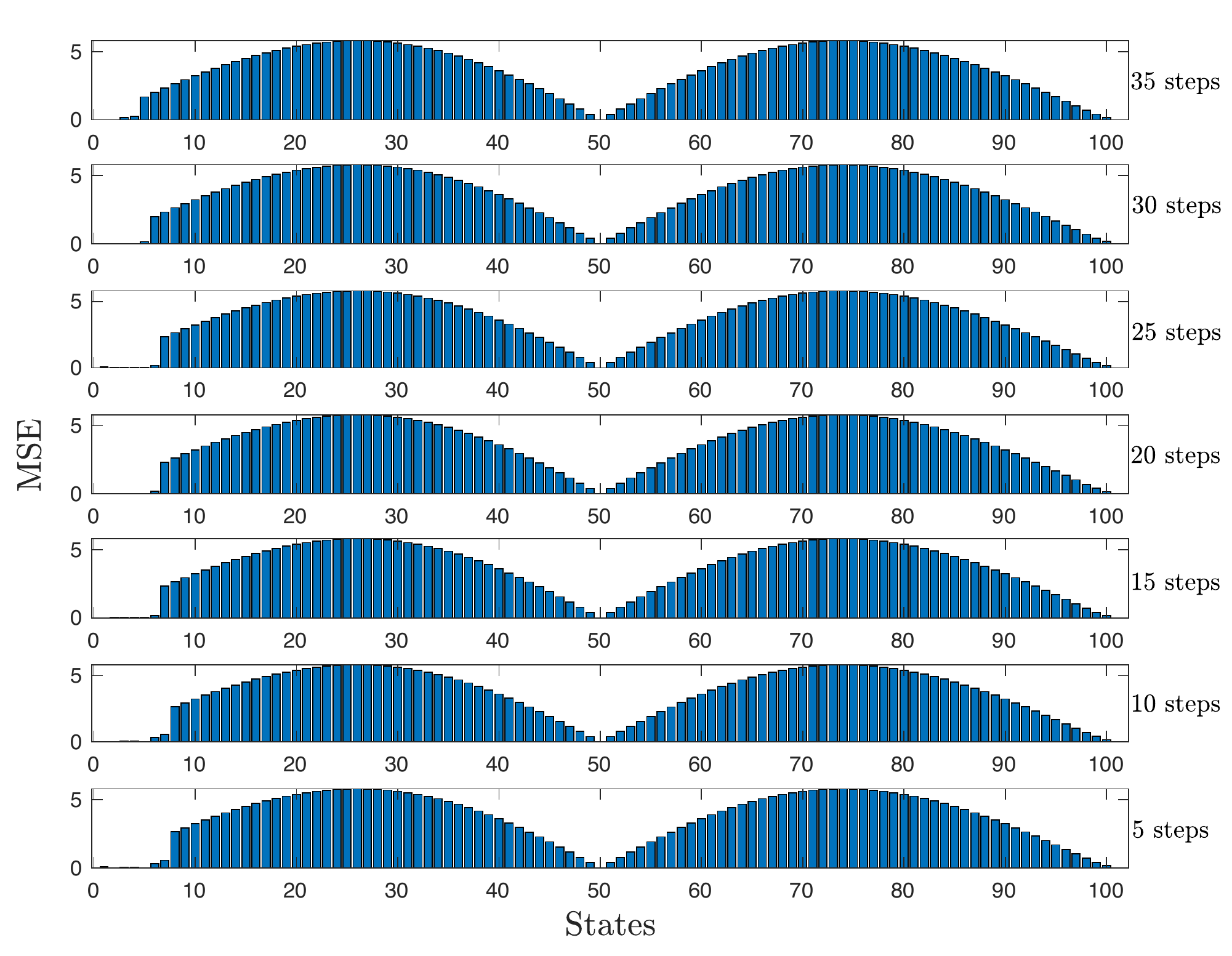}}
\caption{(a) Mean square error of prediction of all the states using Robust DMD aproach. (b) Mean square error of prediction of all the states using normal DMD.}\label{burger_mse}
\end{figure}

Fig. \ref{burger_mse}(a) shows the mean square error in prediction using the proposed approach and Fig. \ref{burger_mse}(b) shows the mean square error using normal DMD. It can be clearly seen that errors using the proposed method are much smaller (of the order of $10^2$). Another observation is that normal DMD is not much sensitive to small variations in training data size, whereas the proposed method is more sensitive to training data size.

\section{Conclusions}\label{section_conclusion}

In this paper, we addressed the problem of computation of Koopman operator from sparse time series data. In certain experimental applications, it may not be possible to obtain time series data which is rich enough to approximate the Koopman operator. We propose an algorithm to compute the Koopman operator for such sparse data. The intuition was based on exploiting the differentiability of the system mapping to append artificial data points to the sparse data set and using robust optimization-based techniques to approximate the Koopman eigenspectrum. The efficiency of the proposed method was also demonstrated on three different dynamical systems and the results obtained were compared to existing Dynamic Mode Decomposition and Extended Dynamic Mode Decomposition algorithms to establish the advantage of our proposed algorithm and in the future; we hope to investigate the performance of our approach on real experimental data sets.

\bibliographystyle{IEEEtran}
\bibliography{subhrajit_robust_DMD}

\end{document}